\documentclass[10pt]{article}
\font\smallit=cmti10
\font\smalltt=cmtt10

\usepackage{amssymb,latexsym,amsmath,epsfig,amsthm,url} 

\makeatletter

\renewcommand\section{\@startsection {section}{1}{\z@}
{-30pt \@plus -1ex \@minus -.2ex}
{2.3ex \@plus.2ex}
{\normalfont\normalsize\bfseries\boldmath}}

\renewcommand\subsection{\@startsection{subsection}{2}{\z@}
{-3.25ex\@plus -1ex \@minus -.2ex}
{1.5ex \@plus .2ex}
{\normalfont\normalsize\bfseries\boldmath}}

\renewcommand{\@seccntformat}[1]{\csname the#1\endcsname. }

\makeatother

\newtheorem{theorem}{Theorem}[section]
\newtheorem{lemma}[theorem]{Lemma}
\newtheorem{corollary}[theorem]{Corollary}

\theoremstyle{definition}

\newtheorem{remark}[theorem]{Remark}
\newtheorem{fact}[theorem]{Fact}

\begin{document}

\begin{center}
\uppercase{\bf Locally nilpotent polynomials over $\mathbb{Z}$}
\vskip 20pt
{\bf Sayak Sengupta}\\
{\smallit Department of Mathematics and Statistics,
		Binghamton University - SUNY,
		Binghamton, New York,
		USA}\\
{\tt sengupta@math.binghamton.edu}\\ 
\vskip 10pt

\end{center}
\vskip 20pt
\vskip 30pt

\centerline{\bf Abstract}
\noindent
For a polynomial $u=u(x)$ in $\mathbb{Z}[x]$ and $r\in\mathbb{Z}$, we consider the orbit of $u$ at $r$ denoted and defined by $\mathcal{O}_u(r):=\{u(r),u(u(r)),\ldots\}$. We ask two questions here: (i) what are the polynomials $u$ for which $0\in \mathcal{O}_u(r)$, and (ii) what are the polynomials for which $0\not\in \mathcal{O}_u(r)$ but, modulo every prime $p$, $0\in \mathcal{O}_u(r)$? In this paper, we give a complete classification of the polynomials for which (ii) holds for a given $r$. We also present some results for some special values of $r$ where (i) can be answered.

\pagestyle{myheadings}
\markright{\smalltt INTEGERS: 23 (2023)\hfill}
\thispagestyle{empty}
\baselineskip=12.875pt
\vskip 30pt

\section{Introduction}
In Example 1 of \cite{B13}, A. Borisov set up a polynomial map, called the \textit{additive trap}, $F_{at}:\mathbb{A}_\mathbb{Z}^2\to \mathbb{A}_\mathbb{Z}^2$, which maps $(x,y)\mapsto (x^2y,x^2y+xy^2)$. This polynomial map satisfies many interesting properties. In particular, $F_{at}^{(p)}(x,y)\equiv (0,0)\pmod p$ for every $(x,y)\in \mathbb{A}^2_{\mathbb{F}_p}$ and for all primes $p$, where $F_{at}^{(p)}$ is the $p$th iteration of $F_{at}$. To prove this, suppose that $p$ is a prime. Note that all points $(x,y)\in\mathbb{A}_{\mathbb{F}_p}^2$ with either $x=0$ or $y=0$ are taken to (0,0) by $F_{at}$.  Let $x\in \mathbb{F}_p^*$. Then for any $y\in\mathbb{F}_p^*$ we get $$\frac{x^2y+xy^2}{x^2y}=\frac{y}{x}+1.$$ So, after at most $p-1$ iterations, the second coordinate becomes 0 and thus, applying $F_{at}$ once more, we reach (0,0). Since $p$ is arbitrary, the proof follows. For more details see \cite{B13}. One can see from the discussion that the $p$th iteration of $F_{at}$ modulo $p$ is the zero map, which follows from the fact that the polynomial $u(x)=x+1$ has the following property: for every $n\in \mathbb{N}$, $u^{(n)}(x)=x+n$, so that, in particular, for every prime $p$, $u^{(p-1)}(1)=p\equiv 0\pmod p$. Throughout this paper, by $\mathbb{N}$ we will mean the set of all positive integers, and by $u^{(n)}$, we will mean $\underbrace{u\circ u\circ \cdots\circ u}_{n\text{ times}}$, the $n$th iteration of $u$. We can write our first definition now which is motivated by the behavior of the polynomial $u(x)=x+1$. Suppose $r\in\mathbb{Z}$ and $A$ is a finite subset of the set of all prime numbers. If, for every prime $p$ not contained in $A$, there exists an $m\in\mathbb{N}$ such that $u^{(m)}(r)\equiv 0\pmod p$, we will say that \textit{$u$ is weakly locally nilpotent at $r$ outside $A$}. The set of all weakly locally nilpotent polynomials at $r$ outside $A$ of degree $d$ will be denoted by $L_{r,A}^d$ and $L_{r,A}$ is the union of $L_{r,A}^d$'s, where the union is taken over all degrees $d\in\mathbb{N}$. When $A=\emptyset$, i.e., for all primes $p$, $m$ exists satisfying $u^{(m)}(r)\equiv 0\pmod p$, we say that $u$ is locally nilpotent at $r$. Thus $u(x)=x+1$ is locally nilpotent at $1$, i.e., $x+1\in L_{1,\emptyset}^1$. If, in particular, a polynomial $u$ is such that $u^{(n)}(r)=0,$ for some $n\in\mathbb{N}$, then we will say that $u$ is \textit{nilpotent at $r$} and the \textit{nilpotency index} is the least of such $n's$. We denote the set of all nilpotent polynomials at $r$ of degree $d$ and nilpotency index $i$ by $N_{r,i}^d$. Also, $N_r$ is the union of all such $N_{r,i}^d$, where the union is taken over $i,d\in\mathbb{N}$. Thus $u(x)=x-1$ is nilpotent at 1 of nilpotency index 1, i.e., $x-1\in N_{1,1}^1$. One of the principal goals of this paper is to understand the polynomials that are locally nilpotent without being nilpotent at $r$, i.e., $0$ is not in the orbit of a polynomial $u$ at $r$ but modulo every prime $p$, some iteration of $u$ at $r$ hits $0$. This is an interesting question and it has been answered completely here. To classify these polynomials we have used Theorem 5 of \cite{TS13}. In particular, if we take $K=\mathbb{Q},g=1,A_1=\{0\},T_1=\{r\},\varphi_1=u\text{ 
 polynomial}$ in this theorem, we get Fact 1.1 in our paper which says:
\begin{fact}
    If $u$ is a polynomial that is of degree at least 2 and it is weakly locally nilpotent at $r$ outside some finite set of primes $A$, then it must be nilpotent at $r$. Equivalently, if there is a polynomial $u$ such that it is weakly locally nilpotent at $r$ outside $A$ but not nilpotent at $r$, then it must be a linear polynomial.
\end{fact}

 It should be noted that Theorem 5 of \cite{TS13} was built upon the work of Silverman in \cite{S93}, and the works of Benedetto-Ghioca-Kurlberg-Tucker-Zannier in Lemma 4.1 of \cite{TT12}. It should also be noted that the question of when an orbit passes through a given point modulo a prime came naturally from the study of the Dynamical Mordell-Lang (DML) Conjecture, which is a major open question in arithmetic dynamics. The paper \cite{TS13} was developed to prove certain special cases of the DML Conjecture. Interested readers can also look at \cite{TT16}.
 
 From Fact 1.1 it is imperative to study the behavior of linear polynomials in order to understand locally nilpotent polynomials which are not nilpotent. This is where we use Theorem 1 of \cite{CS97} and derive a Lemma, which we call CRS Lemma (Lemma 3.2). This Lemma has been stated and proved in section 3.

 The paper contains four main results:
\begin{enumerate}
    \item Complete classification of all polynomials in $L_{r,\emptyset}$, when $r\in\{0,-1,1\}$. This can be found in Theorems 4.1, 4.4 and Corollary 4.2.
    
    \item Complete classification of all polynomials in $L_{1,A}^1$ for any given finite subset $A$ of the set of prime numbers. This can be found in Theorem 5.1. To establish this we have used Lemma 3.2.
    
    \item Complete classification of all polynomials in $S_r$ which can only be linear polynomials by Fact 1.1, where $S_r:=L_{r,\emptyset}\setminus N_r$. This can be found in Corollaries 4.3, 4.5, 5.4 and Theorem 5.3. Lemma 3.2 has also been used to establish this.
\end{enumerate}

 The main tools that we have used here are the following (see section 3 for details):
\begin{enumerate}
    \item Facts 1.1 and 3.1,

    \item Lemma 3.2, 
    
    \item The reduction of polynomials, a technique described after Remark 3.3.
\end{enumerate}

 This paper has 6 sections in total. Section 1 is the introduction. In sections 2 and 3, respectively, we formalize the definitions and introduce the main tools. Section 4 contains the main results listed in (1) above. Section 5 is dedicated to the classification of polynomials in $S_r$. The last section has some open questions and discussions that arose from the study of the polynomials in this paper. Interested reader can also look at the works of Shallit, Vasiga in \cite{VS04} and Odoni in \cite{O85}.

\section{Definitions, notation and terminology}
We will start by formally defining the polynomials mentioned in the introduction and fixing some basic terminology that we will use throughout this paper. Let $\mathcal{P}$ be the set of all positive primes in $\mathbb{Z}$. For a finite subset $A$ of $\mathcal{P}$ and for $a\in \mathbb{Z}$, we define
$$\mathcal{P}_A:=\mathcal{P}\setminus A\;\;\textup{and}\;\; P_A(a):=\{p\in \mathcal{P}_A~|~p \text{ divides }a\}\;\; \textup{and}\;\; P(a):=P_{\emptyset}(a).$$
So $P(a)$ is the set of all positive primes that divides $a$. For $u=u(x)\in\mathbb{Z}[x]$ of degree at least 1, we define the polynomials $u^{(1)}(x):=u(x)$ and $u^{(n+1)}(x):=u(u^{(n)}(x))$, $n\in\mathbb{N}$. Having fixed $r$ in $\mathbb{Z}$, $A$ a finite subset of $\mathcal{P}$, $d$ in $\mathbb{N}$ a degree and $i$ in $\mathbb{N}$ an index, we define the following:

\begin{enumerate}

    \item We will say that $u(x)$ is a \textit{weakly locally nilpotent polynomial} at $r$ outside $A$ if for each $p\in \mathcal{P}_A$, there exists $m\in\mathbb{N}$ (possibly depending on $p$) such that $u^{(m)}(r)\equiv 0\pmod p$. For each $p\in  \mathcal{P}_A$, we let $m_p$ be the least of all such $m'$s. We fix the following notation for weakly locally nilpotent polynomials at $r$ outside $A$:\\ 
    $L_{r,A}^d:=\{u~|~u\textup{, of degree } d,\textup{ is weakly locally nilpotent at }r \textup{ outside }A\},$\\
    $L_{r,A}:=\sqcup_{d=1}^\infty L_{r,A}^d~$.
    
    \item If $A=\emptyset$ in (1), then we will just drop the terms ``weakly" and ``outside $A$". 
    
    \item We will say that $u(x)$ is a \textit{nilpotent polynomial} at $r$ if $\text{there exists an}~ n\in\mathbb{N}$ such that $u^{(n)}(r)=0$. We will call the smallest of all such $n'$s the \textit{nilpotency index/index of nilpotency} of $u(x)$ at $r$. If $u^{(n)}(r)\neq 0$ for all $n\in\mathbb{N}$, we will say that $u$ is \textit{non-nilpotent} at $r$. We fix the following notations for nilpotent polynomials at $r$:\\ $N_{r,i}^d:=\{u~|~u\textup{ is nilpotent at }r \textup{ of nilpotency index }i\textup{ and degree }d\},$\\
    $N_{r,i}:=\sqcup_{d=1}^\infty N_{r,i}^d~,$\\
    $N_r:=\sqcup_{i=1}^\infty N_{r,i}~.$
    
    \item The rest of the notation are as follows:\\
    $S_r:=L_r\setminus N_r$.
    
    For integers $a,b,c~(c\neq 0)$, we will write $a\equiv_c b$ to mean $a\equiv b\pmod c$.
\end{enumerate}

\begin{remark}
It is clear that $N_r\subset L_{r,\emptyset}$. But it turns out that, for every given $r\in\mathbb{Z}$, $S_r$ is non-empty (see Corollaries 4.3, 4.5 and 5.4, and Theorem 5.3 below).

\end{remark}
 
\subsection{Some examples}
\begin{enumerate}
    \item[a.] Let $r\in\mathbb{Z}$. For each non-zero $q(x)\in\mathbb{Z}[x]$, $(x-r)q(x)\in N_{r,1}$.
    
    \item[b.] If $u(x)=-2x-4$, then $u(-1)=-2,~u(-2)=0$. So $u(x)\in N_{-1,2}^1$. If $r\in\mathbb{Z}\setminus\{-1\}$, $u_r(x):=-(r+1)x+(r+1)^2\in N_{r,2}^1$, and if $r\in\mathbb{Z}\setminus\{0\}$, $u_r(x):=-2x+4r\in N_{r,2}^1$. Also if $r\in\mathbb{Z}\setminus\{0\}$, $u(x)=x\pm 1\in N_{r,r}^1$, where we use the negative sign when $r$ is positive and positive sign otherwise.

    \item[c.] Let $u(x)=-2x^2+7x-3$. Then $u(1)=2,~u(2)=3$ and $u(3)=0$. So $u(x)\in N_{1,3}^2$. From this and Fact 3.1 (stated and proved below) it follows that $v(x):=2x^2+7x+3\in N_{-1,3}^2$.

    \item[d.] The polynomial $u(x)=-x^3+9x^2-25x+25\in N_{2,4}^3$.
    
    \item[e.] \textbf{This example shows the existence of non-nilpotent, locally nilpotent polynomials at 1.} Let $u(x)=x+1$. Then by induction it is easy to see that $u^{(n)}(1)=n+1$, for every $n\in\mathbb{N}$, and hence $u\notin N_{1}$. For each $p\in \mathcal{P}$, $u^{(p-1)}(1)=p\equiv_p 0$. Thus $u(x)\in S_1$. In Corollary 4.3 we will see that $S_1=\{x+1\}$.
    
    \item[f.] For every $a\in\mathbb{Z}\setminus\{0\}$, let $u_a=u_a(x):=x+a$. By induction, we get $u_a^{(n)}(0)=na$. So it is clear that $u_a\notin N_0$. For each prime $p$, $u_a^{(p)}(0)=pa\equiv_p 0$. Thus $u_a\in S_0$.
    
    \item[g.] Let $u(x)=4x-2$. Then $u(1)=2$ and $u(2)=6\equiv_5 1$. This means that $u^{(n)}(1)$ is either 1 or 2$\pmod 5$, for every $n\in\mathbb{N}$. This shows that $u(x)\notin L_{1,A}$, for every finite subset $A\subset \mathcal{P}_{\{5\}}$.
   
    \item[h.] Let $u(x)$ be as in \textit{example} (g). Then by induction we have $u^{(n)}(0)=\frac{2}{3}(1-4^n)$,
    which cannot be zero for any $n\in \mathbb{N}$ and so the above polynomial is not in $N_0$. Clearly $m_2=1, ~m_3=3$. For every prime $p\in \mathcal{P}_{\{2,3\}}$, we have that $u^{(p-1)}(0)\equiv_p 0$, by \textit{Fermat's little theorem} and so $u(x)\in S_0$.
 \end{enumerate}

\begin{remark}
The computation of polynomial iterations are very complicated. But the linear polynomials have a nice and easy-to-understand iteration formula: let $u(x)=ax+b$ be a linear polynomial, i.e., $a\in\mathbb{Z}\setminus\{0\}$. Then by induction it follows that for every $n\geq 1$, $$u^{(n)}(x)=a^nx+b\left(\sum\limits_{i=0}^{n-1}a^i\right),~~n\in\mathbb{N}.$$ So $u^{(n)}(r)=a^nr+b\left(\sum\limits_{i=0}^{n-1}a^i\right)$, for each $n\in\mathbb{N}$. Throughout this paper we will refer to
this formula as the \textit{\textbf{linear iteration formula.}}
\end{remark} 

\section{The main tools}
\textit{In this section we will develop the necessary tools. We begin with the proof of Fact 3.1, which indicates that for our purposes it is enough to study the polynomials at non-negative values of $r$. In particular, it shows that there is a one-to-one correspondence between $S_r$ and $S_{-r}$.}

\begin{fact}
Let $u(x)$ be a polynomial of degree $d$ and let $r\in\mathbb{Z}\setminus\{0\}$. Define $v(x):=-u(-x)$. Then $u(x)\in L_{r,\emptyset}^d \iff v(x)\in L_{-r,\emptyset}^d$. Similarly, $u(x)\in N_{r,n}^d \iff v(x)\in N_{-r,n}^d$, and $u(x)\in S_r\iff v(x)\in S_{-r}$.
\end{fact}

\begin{proof}
    Since $v(-x)=-u(x)$, by induction one sees that $v^{(n)}(-r)=-u^{(n)}(r)$, from which the fact follows.\endproof{}
\end{proof}

\vskip 3pt
 We will now state and prove the CRS Lemma which was mentioned in the introduction. We will also justify its importance in understanding linear locally nilpotent polynomials.

\begin{lemma}[CRS Lemma]
Let $\alpha,\beta,\gamma\in\mathbb{Z}\setminus \{0\}$ be such that there is no $k\in\mathbb{Z}$ such that $\frac{\beta}{\gamma}=\alpha^k$. Then $\mathcal{P}\setminus \cup_{n\in\mathbb{N}}P(\gamma \alpha^n-\beta)$ is an infinite set.
\end{lemma} 

\begin{proof}
    Suppose, if possible, that $\mathcal{P}\setminus \cup_{n\in\mathbb{N}}P(\gamma \alpha^n-\beta)$ is a finite set. This means that the set $\cup_{n\in\mathbb{N}}P(\gamma \alpha^n-\beta)$ contains all but finitely many primes. Then $\cup_{n\in\mathbb{N}}P(\gamma \alpha^n-\beta)\setminus P(\gamma)$ also contains all but finitely many primes. So, for almost all $p\in \mathcal{P}_{P(\gamma)}, ~\alpha^{n_p}\equiv_p \beta\gamma^{-1}$ for some $n_p\in \mathbb{N}$ (choice of $n_p$ possibly depends on $p$). So, if $k\in\mathbb{N}$ is such that $\alpha^k\equiv_p 1$, then $(\beta \gamma^{-1})^k\equiv_p (\alpha^{n_p})^k\equiv_p 1$. Thus taking $\alpha=x$, $\beta\gamma^{-1}=y$ and $F=\mathbb{Q}$ in Theorem 1 of \cite{CS97}, we arrive at a contradiction! Thus $\mathcal{P}\setminus \cup_{n\in\mathbb{N}}P(\gamma \alpha^n-\beta)$ is an infinite set.  \endproof{}
\end{proof} 

\begin{remark}[Importance of the CRS lemma]
Let $r\in \mathbb{Z}\setminus\{0\}$ and $u=u(x)=ax+b\in L_{r,\emptyset}^1$ with $a\neq \pm 1$. By the \textit{linear iteration formula}, we get
$$u^{(n)}(r)=\frac{a^n(r-ar-b)+b}{1-a}.$$ Since $u\in L_{r,\emptyset}^1$, we can say that $\mathcal{P}\setminus \cup_{n\in\mathbb{N}} P(\gamma \alpha^n-\beta)$ is a finite set (in fact, it is an empty set), where $\alpha=a,\beta=-b$ and $\gamma=r-ar-b$. Then it follows from the Lemma 3.2 that either $\frac{\beta}{\gamma}$ or $\frac{\gamma}{\beta}$ is a power of $\alpha$, i.e., $b=-a^m(r-ar-b),$ for some $m\in\mathbb{Z}$. Moreover, if $m\in\mathbb{N}$ we can say that $u\in N_r$. So, to summarize, if $u$ is in $S_r$ (with $a\notin\{ \pm 1\}$), then $\text{there exists}~m\in\mathbb{N}\cup \{0\}$ such that $a^mb=b+ar-r$.
\end{remark}

\subsubsection*{Reduction of polynomials.} Let $r\in \mathbb{N}$ and $u=u(x)\in\mathbb{Z}[x]$ such that $r|u(0)$. Define $v=v(x):=\frac{1}{r}u(rx)$. Note that $v$ is a polynomial over $\mathbb{Z}$ of the same degree as $u$ and $rv(1)=u(r)$. Using induction one can see that $rv^{(n)}(1)=u^{(n)}(r),\text{ for all} ~n\in\mathbb{N}$. Then it follows that $u(x)$ is weakly locally nilpotent at $r$ outside some $A$ iff $v(x)$ is weakly locally nilpotent at 1 outside $A\cup P(r)$, and also $u(x)$ is nilpotent at $r$ iff $v(x)$ is nilpotent at 1. Thus we can reduce any polynomial $u(x)$ in $L_{r,\emptyset}^d$ with the extra condition that $r|u(0)$ to a polynomial $v(x)$ in $L_{1,P(r)}^d$. We will call this the \textit{reduction of $u(x)$ to $v(x)$}, where, of course, $u$ and $v$ are as above.

\section{Arbitrary $d$ and $r\in\{0,1,-1\}$}

In this section, we state and prove two theorems that provide the classification of all locally nilpotent polynomials at $r$ when $r\in\{0,\pm 1\}$. We start with the $r=1$ case.

\begin{theorem}
The following is the list of all polynomials in $L_{1,\emptyset}$\textup{:}
\begin{enumerate}
    \item[(1)] $ (x-1)p(x)$, with $p(x)\in\mathbb{Z}[x]\setminus\{0\}$ (\textit{Nilpotent of nilpotency index 1}).
    
    \item[(2)] $ -2x+4+p(x)(x-1)(x-2)$, with $p(x)\in\mathbb{Z}[x]$ (\textit{Nilpotent of nilpotency index 2}).
    
    \item[(3)] $ -2x^2+7x-3+p(x)(x-1)(x-2)(x-3)$, with $p(x)\in\mathbb{Z}[x]$ (\textit{Nilpotent of nilpotency index 3}).
    
    \item[(4)] $ x+1$ (\textit{Locally nilpotent but not nilpotent}).
\end{enumerate}
\end{theorem}
\begin{proof}
    
Let $u=u(x)\in L_{1,\emptyset}^d$. We will consider the following three cases:

\subsubsection*{Case 1. $u(1)-1\not \in\{\pm 1\}$.}
Then $P(u(1)-1)\neq\emptyset$ and for each $p\in P(u(1)-1)$ we have $u(1)\equiv_p 1$, i.e., for each $p\in P(u(1)-1)$, $m_p$ does not exist. This is a contradiction to the hypothesis that $u\in L_{1,\emptyset}^d$!

\subsubsection*{Case 2. $u(1)-1=-1$.}
These are just the polynomials listed in (1).

\subsubsection*{Case 3. $u(1)-1=1\textup{ or }u(1)=2$.}
We now explore the possibilities for $u(2)$. If $u(2)=0$, then $u(x)$ is of the form listed in (2). So suppose that $u(2)\neq 0$. Of course $u(2)\notin\{1,2\}$ as otherwise we get $u^{(n)}(1)=$ 1 or 2, for every $n\in\mathbb{N}$ and hence it cannot be in $L_{1,\emptyset}^d$. Thus $u(2)$ is either $\le -1$ or $\ge 3$, i.e., $|u(2)-1|\ge 2$. In other words, $P(u(2)-1)\neq \emptyset$. Let $p\in P(u(2)-1)$. Then $u(2)\equiv_p 1$. As $u$ is locally nilpotent at $1$, $p$ must be 2 so that $u(2)-1$ must be of the form $\pm 2^t$, for some $t\in\mathbb{N}$. To arrive at a contradiction suppose that $u(2)\neq 3$. That means $u(2)$ is either $\geq 4$ or $\leq -1$. So we consider these two possibilities one by one.

\textbf{Possibility 1}. $u(2)\ge 4$. We know that $u(2)-1=2^t$, i.e., $u(2)$ is odd. So, in fact, $u(2)\ge 5$. Thus there exists $p\in \mathcal{P}_{\{2\}}$ such that $p\in P(u(2)-2)$ and $u^{(n)}(2)\equiv_p 2$, for every $n\in\mathbb{N}$, a contradiction to the hypothesis that $u\in L_{1,\emptyset}^d$!

\textbf{Possibility 2}. $u(2)\le -1$. We know that $u(2)-1=-2^t$, i.e., $u(2)$ is odd so that $u(2)-2$ is odd as well and less or equal to $-3$. Using the same argument as in possibility 1 we get a contradiction!

 Thus $u(2)$ must be 3. Next, we look at $u(3)$. If $u(3)=0$, then $u(x)$ is of the form listed in (3). So suppose that $u(3)\neq 0$. For the same reason as above $u(3)\notin\{0,1,2,3\}$. Thus $u(3)$ is either $\le -1$ or $\ge 4$. To get to a contradiction suppose that $u(3)\neq 4$. Then either $u(3)-3\le -4$ or $\geq 2$. In any case, $P(u(3)-3)\neq\emptyset$. Let $p\in P(u(3)-3)$. Then $u(3)\equiv_p 3$ so that $p\in \{2,3\}$, which follows because $u\in L_{1,\emptyset}^d$. If $p=2$ then $u(1)\equiv_p 0$. Then $u(3)\equiv_p 3\equiv_p 1\not\equiv_p u(1)$, which is an impossibility as $3\equiv_p 1$ must imply $u(3)\equiv_p u(1)$! This means $p=3$ and $u(3)-3=\pm 3^s$, for some $s\in\mathbb{N}$. For similar reason, $P(u(3)-1)\neq \emptyset$ and for each $q\in P(u(3)-1)$, $u(3)\equiv_q 1$ which implies $q\in\{2,3\}$. But $q|u(3)-1=2\pm 3^s$, which is a contradiction to the fact that $q\in\{2,3\}$! Thus $u(3)=4$.

 Next, we look at $u(4)$. We claim that no further iteration of $u$ at $1$ can be zero and we would like to prove this by showing that $u(n-1)=n$, $\text{ for all} ~n\ge 4$ and that would mean $u(x)=x+1$. We use \textit{mathematical induction} to prove this claim. Let $u(j-1)=j$, for every $2\le j\le n$, for some $n\ge 4$ and we want to show that $u(n)=n+1$. Since $u(1)=2, ~u(2)=3, ~u(3)=4,~\ldots, u(n-1)=n$, there is a polynomial $p(x)$ such that $u(x)=x+1+p(x)(x-1)(x-2)(x-3)\cdots (x-n+1)$. So $u(n)=n+1+p(n)\cdot (n-1)!$, which must be different from 0 as $n\ge 4$. If $u(n)=i$ for some $i\in\{1,\ldots,n\}$, then the iterations $u^{(m)}(1)\in\{1,\ldots,n\}$, for every $m\in\mathbb{N}$ and that means for only finitely many primes $p$, $m_p$ can exist. Thus $u$ cannot be locally nilpotent at $1$ and $u(n)\not \in \{0,\ldots,n\}$. This means $u(n)$ is either $\ge n+1$ or $\le -1$. For a contradiction, suppose that $u(n)\neq n+1$. Then, either $u(n)-n\ge 2$ or $u(n)-n\le -(n+1)$. In any case, we get $P(u(n)-n)\neq\emptyset$. For each $p\in P(u(n)-n)$ we have $u(n)\equiv_p n$, which is an impossibility unless $p\le n$. Suppose, if possible, $p<n$. Then $n\equiv_p a$ for some $a\in \{1,\ldots,p-1\}$. Note that $a$ cannot be zero as otherwise $n\equiv_p 0$, so that $u(0)\equiv_p u(n)\equiv_p n\equiv_p 0$, i.e., $p|u(0)$ and so $p|u(p)=p+1$, an impossibility! Now by the induction hypothesis we have $u(a)=a+1$ and also $a\equiv_p n\equiv_p u(n)\equiv_p u(a)$, i.e., $u(a)\equiv_p a$, which is absurd as this means $m_p$ does not exist! So $p=n$, i.e., $n$ is prime and $u(n)=n\pm n^s$, for some $s\in\mathbb{N}$. For similar reason, $P(u(n)-1)\neq \emptyset$. So for every $q\in P(u(n)-1)$, $u(n)\equiv_q 1$ and it follows that $q$ is less than or equal to $n$. But if $q=n$, then $n=q|u(n)-1=(n-1)\pm n^s$ and so $n|1$, an impossibility! So, in fact, we have $q\le n-1$. We can choose $b\in\{0,\ldots,q-1\}$ such that $n\equiv_q b+1$. By the induction hypothesis $u(b+1)=b+2$ and also $u(b+1)\equiv_q u(n)\equiv_q 1$. These two relations together imply $b+1\equiv_q 0$, i.e., $q|n$. But, since $n$ is a prime, $n=q$, which is an impossibility, and hence $u(n)=n+1$.\endproof{}
\end{proof}

\begin{corollary}
It follows from Fact 3.1 and Theorem 4.1 that the following is the list of all polynomials in $L_{-1,\emptyset}$:
\begin{enumerate}
    \item[(1)] $ (x+1)p(x)$, with $p(x)\in\mathbb{Z}[x]\setminus\{0\}$ (Nilpotent of nilpotency index 1).
    
    \item[(2)] $ -2x-4+p(x)(x+1)(x+2)$, with $p(x)\in\mathbb{Z}[x]$ (Nilpotent of nilpotency index 2).
    
    \item[(3)] $ 2x^2+7x+3+p(x)(x+1)(x+2)(x+3)$, with $p(x)\in\mathbb{Z}[x]$ (Nilpotent of nilpotency index 3).
    
    \item[(4)] $ x-1$ (Locally nilpotent but not nilpotent).
\end{enumerate}
\end{corollary}

\begin{corollary}
The sets $S_1$ and $S_{-1}$ are singleton sets.
\end{corollary} 

\begin{proof}
    Let $u(x)\in S_1$. Then by Theorem 4.1, $u(x)$ must be $x+1$ as all the other polynomials in the list (1)-(4) in there are in $N_1$. Now by Fact 3.1, it follows that $S_{-1}=\{x-1\}$. \endproof{}
\end{proof}   

 We end this section by looking at the case when $r=0.$

\begin{theorem}
The following is the list of all polynomials in $L_{0,\emptyset}$\textup{:}
\begin{enumerate}
    \item[(1)] $x+b$, with $b\in\mathbb{Z}\setminus \{0\}$ (\textit{Locally nilpotent but not nilpotent}).
    
    \item[(2)] $ ax+b$, with $P(b)\supseteq P(a)\neq \emptyset$ and $b\neq 0$ (\textit{Locally nilpotent but not nilpotent}). 

    \item[(3)] $xp(x)$, with $p(x)\in\mathbb{Z}[x]\setminus\{0\}$ (\textit{Nilpotent of nilpotency index 1}).

    \item[(4)] $(x-a)p(x)$, with $a\in\mathbb{Z}\setminus\{0\}$ and $p(x)\in \mathbb{Z}[x]$ s.t $p(0)=-1$ (\textit{Nilpotent of nilpotency index 2}).
    
\end{enumerate}
\end{theorem}

\begin{proof}
First suppose that $u$ is nilpotent of nilpotency index $m$, for some $m\in\mathbb{N}$. If $u(0)=0$, then $m=1$ and $u(x)=xp(x),$ for some non-zero $p(x)\in\mathbb{Z}[x]$, which is (3) in the list. So suppose that $u(0)\neq 0$. Define $$u_0:=u(0),~u_n:=u^{(n+1)}(0)-u^{(n)}(0),~n\in\mathbb{N}.$$ Then $u_{n+1}=u^{(n+2)}(0)-u^{(n+1)}(0)= u(u^{(n+1)}(0))-u(u^{(n)}(0)).$ That means $u_n$ divides $u_{n+1},\text{ for all}~n\in\mathbb{Z}_{\ge 0}$. We also have $u^{(m)}(0)=0$ so that $u_m=u^{(m+1)}(0)-u^{(m)}(0)=u^{(m+1)}(0)=u_0$. 
As $u_0|u_1|\ldots|u_m=u_0,$ it follows that $u_n=\pm u_0,$ for all $n.$ Also note that $u_0+\cdots+u_{m-1}=u^{(m)}(0)=0$. This means $m$ must be even and half of these integers is positive and the other half is negative (since $|u_n|=|u_0|,\text{ for all}~n$). So there exists $k\in\{1,\ldots,m-1\}$ such that $u_{k-1}=-u_k,$ i.e., $u^{(k)}(0)-u^{(k-1)}(0)=u^{(k)}(0)-u^{(k+1)}(0),$ i.e., $u^{(k+1)}(0)=u^{(k-1)}(0)$. Thus $u^{(n+2)}(0)=u^{(n)}(0),\text{ for all} ~n\ge k-1$ and so, in particular, $0=u^{(m)}(0)=u^{(m+2)}(0)=u^{(2)}(0)$. Hence, $m=2$ and $u(x)=(x-a)p(x)$, with $a\in\mathbb{Z}\setminus\{0\}$ and $p(x)\in\mathbb{Z}[x]$ with $p(0)=-1$; here $a=u(1),$ which is (4) in the list.

 Now suppose that $u\in S_0$. Then by Fact 1.1, it must be linear. Let $u(x):=ax+b, a\neq 0$. Note that if $b\neq0$ as otherwise $u$ would be nilpotent. When $a=1$, every $u(x)\in S_0$: in fact if $u(x)=x+b$, then by the \textit{linear iteration formula}, $u^{(n)}(0)=b(1+\cdots+1)=bn$, which is non-zero $\text{ for all}~n\in\mathbb{N}$ and for each prime $p$, $u^{(p)}(0)=bp\equiv_p 0$. When $a=-1$, $u^{(2)}(0)=0$. So $a$ must be different from $-1$. Thus we can assume that $|a|\geq 2$, i.e., $P(a)$ is a non-empty, finite set. Again, using the \textit{linear iteration formula}, we get $u^{(n)}(0)=b(1+\cdots+a^{n-1}),~n\in\mathbb{N}$. Suppose, if possible, there exists some prime $p$ in $P_{P(b)}(a)$, i.e., there is a prime $p$ such that $p|a$ but $p\nmid b$. Then $u^{(n)}(0)\equiv_p b$, for every $n\in\mathbb{N}$ and that means $u$ cannot be locally nilpotent. Thus $P(b)\supseteq P(a)\neq\emptyset$. If $p\in P(b)$, then $m_p=1$.\\
If $p\notin P(b)\cup P(a-1)$, 
then $u^{(p-1)}(0)=\frac{b}{a-1}(a^{p-1}-1)\equiv_p 0$, which follows from \textit{Fermat's little theorem}.\\
Finally if $p\in P(a-1)$, then $u^{(p)}(0)=b(1+\cdots+a^{p-1})\equiv_p b(1+\cdots+1)\equiv_p 0$. Thus $m_p$ exists for every $p\in\mathcal{P}$.\endproof{}
\end{proof}

\begin{corollary}
The following is the list of all polynomials in $S_{0,\emptyset}$\textup{:}
\begin{enumerate}
    \item[(1)] $x+b$, with $b\in\mathbb{Z}\setminus \{0\}$.
    
    \item[(2)] $ ax+b$, with $P(b)\supseteq P(a)\neq \emptyset$ and $b\neq 0$.
\end{enumerate}
\end{corollary}

\section{Linear case $d=1$}

In this section we classify all polynomials that are locally nilpotent but non-nilpotent at $r$. Since the cases $r=0$ and $r=\pm 1$ have been covered in the previous section, here our focus would be on the values of $r$ that lie in $\mathbb{Z}\setminus\{0,\pm 1\}$. First, we state and prove a classification theorem for polynomials which are weakly locally nilpotent at 1 outside some given finite subset $A$ of $\mathcal{P}$. We will use this theorem to prove our final main result, Theorem 5.3.

\begin{theorem}
Let $A=\{q_1,\ldots, q_k\}$, where $q_1,\ldots, q_k$ are $k$ distinct primes. Then the following is the list of all the polynomials in $L_{1,A}^1$\textup{:}
\begin{enumerate}
    \item[(1)] $ x\pm q_1^{s_1}\cdots q_k^{s_k}$, where $s_i\in\mathbb{N}\cup\{0\}$.
    
    \item[(2)] $ \alpha(x-1)$, $\alpha\in\mathbb{Z}\setminus\{0\}$.
    
    \item[(3)] $ \pm q_1^{s_1}\cdots q_k^{s_k} x+1$, where $s_i\in\mathbb{N}\cup\{0\}$ such that $\sum s_i\ge 1$.
    
    \item[(4)] $ -2x-1$ (only when $2\in A$).
    
    \item[(5)] $ -2x+4$.
\end{enumerate}
\end{theorem}

\begin{proof}
Let $u=u(x)\in L_{1,A}^1$. Then by Fact 1.1 it must be linear, say $ax+b$. It is clear that $b\neq 0$ as otherwise $u^{(n)}(1)$ would just be $a^n$ and it cannot be divisible by any prime $p\in \mathcal{P}_{P(a)}$. By the \textit{linear iteration formula}, $u^{(n)}(1)=a^n+b(1+\cdots+a^{n-1})$, for every $n\in\mathbb{N}$. 
\begin{align*}
    \text{Note that if } a=1, ~u^{(m_p)}(1)& = 1+bm_p\equiv_p 0, \text{ for every prime }p\notin A,\\
    \implies & bm_p\equiv_p -1, \text{ for every prime }p\notin A,\\
    \implies & b\text{ is invertible in }\mathbb{F}_p, \text{ for every prime }p\notin A,\\
    \implies & b=\pm q_1^{s_1}\cdots q_k^{s_k}, \text{ for some }s_i's \textup{ in }\mathbb{N}\cup \{0\}.\hspace{57pt}
\end{align*}
One can check that the polynomials $x\pm q_1^{s_1}\cdots q_k^{s_k}$ are indeed in $L_{1,A}^1$, for $s_i\in\mathbb{N}\cup\{0\}$.

 If $a=-1$, $u(x)=-x+b$ and $u^{(2)}(x)=x$. So $u$ cannot be in $L_{1,A}^1$  unless $b=1$, and in that case, it is in fact in $N_{1,1}^1$. Thus we can assume that $|a|\ge 2$. Similar to the proof of Theorem 4.1, we can break down these polynomials into the following three cases:

\subsubsection*{Case 1. $u(1)-1\notin \{\pm 1\}$.}
This means that $P(u(1)-1)\neq \emptyset$. So $a+b=1 \pm q_1^{s_1}\cdots q_k^{s_k}$, i.e., $b=1-a \pm q_1^{s_1}\cdots q_k^{s_k}$, for some $s_i\in\mathbb{N}\cup\{0\}$ with $\sum s_i\ge 1$. Then by the \textit{linear iteration formula}, we have $$u^{(n)}(1)=\frac{b\pm a^n(1-a-b)}{1-a}$$ and it follows from Remark 3.3 that $\text{there exists an }m\in\mathbb{Z}$ such that $b=\pm a^m(1-a-b)$. If $m=0$, then $b=\pm(1-a-b)$, i.e., $a+2b=1$ or $a=1$. Since $|a|\ge 2$, we deduce that $a+2b=1$.We also have $a+b=1\pm q_1^{s_1}\cdots q_k^{s_k}$. Solving $a$ and $b$ from these two equations we get $a=1\pm 2q_1^{s_1}\cdots q_k^{s_k},~b=\mp q_1^{s_1}\cdots q_k^{s_k}$. So $u(x)=(1\pm 2q_1^{s_1}\cdots q_k^{s_k})x\mp q_1^{s_1}\cdots q_k^{s_k}$ and so $u^{(n)}(1)=\frac{1+(1\pm 2q_1^{s_1}\cdots q_k^{s_k})^n}{2},~n\in\mathbb{N}$. Letting $\alpha=1\pm 2q_1^{s_1}\cdots q_k^{s_k},\beta=-1$ and $\gamma =1$, it is clear that neither $\frac{\beta}{\gamma}$ nor $\frac{\gamma}{\beta}$ is a power of $\alpha$. So it follows from the Lemma 3.2 that $(1\pm 2q_1^{s_1}\cdots q_k^{s_k})x\mp q_1^{s_1}\cdots q_k^{s_k}\notin L_{1,A}^1$. Thus $m$ must be a non-zero integer. We consider the following four possibilities:

 \underline{If $m\in\mathbb{N}$ and $b=a^m(1-a-b)$}, then $b(1+a^m)=a^m(1-a)$. Since $\gcd(a^m,a^m+1)=1$, we must have $a^m+1|1-a$ which is only possible if $m=1$.

\underline{If $m\in\mathbb{N}$ and $b=-a^m(1-a-b)$}, then $b(1-a^m)=-a^m(1-a)$, i.e., $b(1+\cdots+a^{m-1})=-a^m$. Since $\gcd(1+\cdots+a^{m-1},a^m)=1$, we must have $1+\cdots+a^{m-1}=\pm 1$ which is only possible if $m\in\{1,2\}$.

\underline{If $m=-n,\text{ with }n\in\mathbb{N}$ and $b=a^m(1-a-b)$}, then $ba^n=1-a-b$, i.e., $b(a^n+1)=1-a$. It follows from above that this is only possible if $n=1$.

\underline{If $m=-n,\text{ with }n\in\mathbb{N}$ and $b=-a^m(1-a-b)$}, then $ba^n=-(1-a-b)$, i.e., $b(a^n-1)=a-1$. Again using the same logic as above, we conclude that $n\in\{1,2\}$.

Thus we only need to look at $m=\pm 1,\pm 2$, which we investigate in the following four subcases:

\noindent\textbf{Subcase 1.} $m=-1$.\\
Here we have $ba=\pm(1-a-b)$. First suppose that $ba=1-a-b$, i.e., $b(a+1)=1-a$. This means that $a+1|a-1$ and this is only possible if $a=-2$ and $a=-3$. These values generate the polynomials $u(x)=-2x-3$ and $u(x)=-3x-2$, respectively. When $u(x)=-2x-3$, the \textit{linear iteration formula} gives $$u^{(n)}(1)=2(-2)^n-1.$$ Letting $\alpha=-2,\beta=1$ and $\gamma=2$, it is clear that neither $\frac{\beta}{\gamma}$ nor $\frac{\gamma}{\beta}$ is a power of $\alpha$. So, by the Lemma 3.2, $-2x-3\notin L_{1,A}^1$. Similarly we can show that $-3x-2\notin L_{1,A}^1.$

 Now suppose $ba=-(1-a-b)$. This gives $b=1$ and hence $a=\pm q_1^{s_1}\cdots q_k^{s_k}$. Thus $u(x)=\pm q_1^{s_1}\cdots q_k^{s_k}x+1$ and it follows from the \textit{linear iteration formula} that $$u^{(n)}(1)=(\pm q_1^{s_1}\cdots q_k^{s_k})^n+[1+\cdots+(\pm q_1^{s_1}\cdots q_k^{s_k})^{n-1}]=\frac{1-(\pm q_1^{s_1}\cdots q_k^{s_k})^{n+1}}{1-(\pm q_1^{s_1}\cdots q_k^{s_k})}, ~n\in\mathbb{N}.$$ If $p\in P_A(1-(\pm q_1^{s_1}\cdots q_k^{s_k}))$, then $u^{(p)}(1)\equiv_p p\equiv_p 0$. So suppose that $p\notin P_A(1-(\pm q_1^{s_1}\cdots q_k^{s_k}))$. Now if $2\in A$, then the existence of $m_2$ is not a concern and if $2\notin A$, then $2\in P_A(1-(\pm q_1^{s_1}\cdots q_k^{s_k}))$ which was covered above. Finally, if $p$ is in ${A\cup \{2\}}$, then $u^{(p-2)}(1)\equiv_p 0$, by \textit{Fermat's little theorem}. So $u(x)=\pm q_1^{s_1}\cdots q_k^{s_k}x+1$ is in $L_{1,A}^1$.

\noindent\textbf{Subcase 2.} $m=1$.\\
Here we have $b=\pm a(1-a-b)$. First suppose that $b=a(1-a-b)$, i.e., $b(a+1)=a(1-a)$. The same reasoning as above implies $a+1|a-1$ so that the only possibilities we get are $a=-2,~b=6$ or $a=-3,~b=6$. These values produce the polynomials $u(x)=-2x+6$ and $u(x)=-3x+6$, respectively. When $u(x)=-2x+6$, the \textit{linear iteration formula} gives $$u^{(n)}(1)=-(-2)^n+2.$$ Letting $\alpha=-2,\beta=-2$ and $\gamma=-1$, it is clear that neither $\frac{\beta}{\gamma}$ nor $\frac{\gamma}{\beta}$ is a power of $\alpha$. So, by the Lemma 3.2, $-2x+6\notin L_{1,A}^1$. Similarly, we can show that $-3x+6\notin L_{1,A}^1$.

\noindent\textbf{Subcase 3.} $m=-2$.\\
Here we have $ba^2=\pm (1-a-b)$. First suppose $ba^2=1-a-b$, i.e., $b(a^2+1)=1-a$. This means that $a^2+1|a-1$ which is not possible as $|1-a|\le 1+|a|<1+a^2$. Thus $ba^2=-(1-a-b)$, i.e., $b(a+1)=1$, i.e., $b=a+1=\pm 1$. So $u(x)=-2x-1$. It follows from the \textit{linear iteration formula} that $$u^{(n)}(1)=(-2)^n-[1+\cdots+(-2)^{n-1}]=\frac{(-2)^{n+2}-1}{3},~n\in\mathbb{N}.$$ It is easy to see that $m_2$ does not exist, $m_3=1$ and for all $p\in \mathcal{P}_{\{2,3\}}$, $u^{(p-3)}(1)\equiv_p 0$. So $-2x-1$ is in $L_{1,A}^1$ iff $2\in A$.

\noindent\textbf{Subcase 4.} $m=2$.\\
Here we have $b=\pm a^2(1-a-b)$. First suppose that $b=a^2(1-a-b)$, i.e., $b(1+a^2)=a^2(1-a)$. Since $\gcd(1+a^2,a^2)=1$, we must have $1+a^2|1-a$, which is impossible (see the subcase 3 above). So $b=-a^2(1-a-b)$, i.e., $b(a+1)=-a^2$ which means $a+1=\pm 1$ and $b=\mp a^2$. Since $|a|\ge 2$, this means $a=-2$ and $b=4$. But then $u(1)-1=1\in \{\pm 1\},$ an impossibility in this \textit{case}!

This wraps us case 1. Now we look at the remaining two cases.

\subsubsection*{Case 2. $u(1)-1=-1$, i.e., $u(1)=0$.}
These are the polynomials in $N_{1,1}^1$.

\subsubsection*{Case 3. $u(1)-1=1$, i.e., $u(1)=2$.}
If $u(2)=0$, then $u(x)=-2x+4$. So we can suppose that $u(2)\notin\{0,1,2\}$, i.e., $u(2)$ is either $\le -1$ or $\ge 3$, i.e., $|u(2)-1|\ge 2,$ i.e., $P(u(2)-1)\neq \emptyset$. If $u(2)=3$, then $u(x)=x+1\in S_1$. So we can further assume that $u(2)\neq 3$. Since $u(1)=2,~b=2-a$ and so $u(x)=ax+(2-a)$. Then by the \textit{linear iteration formula}, we get
$$u^{(n)}(1)=\frac{2-a-a^n}{1-a},~~n\in\mathbb{N}.$$
Since $u\in L_{1,A}^1$ it follows from Remark 3.3 that $2-a=a^m$, for some $m\in\mathbb{Z}$.
Note that $m$ cannot be a non-positive integer as otherwise it would follow that $a$ must be equal to $\pm 1$. Thus $m\in\mathbb{N}$ and $2=a(1+ a^{m-1})$. Therefore $a=\pm 2$ and $1+a^{m-1}=\pm 1,$ i.e., $a^{m-1}=-2$, i.e., $a=-2$. But $a=-2$ implies that $b=4$ and hence $u(2)=2a+b=0$, which cannot happen as we have already said that $u(2)\not\in\{0,1,2,3\}$. Thus $ax+(2-a)\notin L_{1,A}^1$. \endproof{}
\end{proof}

 The next corollary follows directly from the computations in the proof of Theorem 5.1:

\begin{corollary}
Let $A=\{q_1,\ldots,q_k\}$ and $q_1,\ldots,q_k$ be $k$ distinct primes. Then the following is the list of all polynomials in $L_{1,A}^1\setminus N_1$\textup{:}
\begin{enumerate}
    \item[(1)] $ x+ q_1^{s_1}\cdots q_k^{s_k}$, where $s_i\in\mathbb{N}\cup\{0\}$.
    
    \item[(2)] $ x- q_1^{s_1}\cdots q_k^{s_k}$, where $s_i\in\mathbb{N}\cup\{0\}$ such that $\sum s_i\ge 1$.
    
    \item[(3)] $ \pm q_1^{s_1}\cdots q_k^{s_k} x+1$, where $s_i\in\mathbb{N}\cup\{0\}$ such that $\sum s_i\ge 1$.
    
    \item[(4)] $ -2x-1$ (only when $2\in A$).
\end{enumerate}
\end{corollary}

 Finally we state and prove the last (main) result of this paper.

\begin{theorem}
Let $r$ be a natural number greater than or equal to 2 and $r=q_1^{a_1}\cdots q_k^{a_k}$ be the prime decomposition. Then the following is the list of all polynomials in $S_r$\textup{:}
\begin{enumerate}
    \item[(1)] $x+ q_1^{s_1}\cdots q_k^{s_k}$, where $s_i\in\mathbb{N}\cup\{0\}$.
    
    \item[(2)] $x-q_1^{s_1}\cdots q_k^{s_k}$, where $s_i\in\mathbb{N}\cup\{0\}$ with at least one $j\in\{1,\ldots, k\}$ s.t. $s_j>a_j$.
    
    \item[(3)] $\pm q_1^{s_1}\cdots q_k^{s_k}x+r$, where $s_i\in\mathbb{N}\cup\{0\}$ with $\sum\limits_{i} s_i\ge 1$.
    
    \item[(4)] $-2x-r$ (only when $r$ is even).
\end{enumerate}
\end{theorem}

\begin{proof}
Let $u=u(x)\in S_r$ and $A:=P(r)$. Then by Fact 1.1 $u$ must be linear, say $ax+b$. First we will look at the instances when $a=\pm 1$.
\begin{align*}
    \text{Note that if } a=1, ~u^{(m_p)}(1)=&1+bm_p\equiv_p 0, \text{ for every prime }p\notin A,\\
    \implies & bm_p\equiv_p -1, \text{ for every prime }p\notin A,\\
    \implies & b\text{ is invertible in }\mathbb{F}_p, \text{ for every prime }p\notin A,\\
    \implies & b=\pm q_1^{s_1}\cdots q_k^{s_k}, \text{ for some }s_i's \textup{ in }\mathbb{N}\cup \{0\}.\hspace{57pt}
\end{align*}
So, when $a=1$, we expect $u(x)$ to be of the form $x \pm q_1^{s_1}\cdots q_k^{s_k}$. First suppose that $u(x)=x+q_1^{s_1}\cdots q_k^{s_k}$. Then by the \textit{linear iteration formula}, $u^{(n)}(r)=q_1^{a_1}\cdots q_k^{a_k}+n\cdot q_1^{s_1}\cdots q_k^{s_k}$, which is always non-zero for every $n\in\mathbb{N}$. One can check that these polynomials are locally nilpotent and they are in (1) in the list above. Now suppose that $u(x)=x-q_1^{s_1}\cdots q_k^{s_k}$. Then by the \textit{linear iteration formula} $u^{(n)}(r)=q_1^{a_1}\cdots q_k^{a_k}-n\cdot q_1^{s_1}\cdots q_k^{s_k},~n\in\mathbb{N}$. If, for all $i\in\{1,\ldots,k\}$, $s_i\le a_i$, then $u^{(q_1^{a_1-s_1}\cdots q_k^{a_k-s_k})}(r)=0$, which is a contradiction as $u$ is non-nilpotent! That means we must have at least one $j\in\{1,\ldots,k\}$ such that $a_j<s_j$. One can now check that $u$ is non-nilpotent but locally nilpotent and so we get the polynomials in (2) in the list above.

 If $a=-1$, $u(x)=-x+b$ and $u^{(2)}(x)=x$. So $u$ cannot be in $L_{r,\emptyset}^1$  unless $b=r$, and in that case, it is in fact in $N_{r,1}^1$, a contradiction! So $a\neq -1$. Thus we suppose $|a|\ge 2$. It follows from Remark 3.3 that $\text{there exists an }~m\in\mathbb{N}\cup \{0\}$ such that $a^mb=b+ar-r$. If $m=0$ then $r(1-a)=0$ which is an impossibility as $r\neq 0$ and $|a|\ge 2$. That means that $m\in\mathbb{N}$ and $b(a^m-1)=r(a-1),$ i.e., $b(1+\cdots+a^{m-1})=r$ so that $b|r$. 

 We want to show that $u(r)-r\notin \{\pm 1\}$. Suppose otherwise, i.e., $u(r)=r\pm 1$. This means that $b=r-ar\pm 1$, i.e., $u(x)=ax+(r-ar\pm 1)$. We will only consider the possibility $b=r-ar-1$ as the other possibility can be rejected using the same argument. Applying the \textit{linear iteration formula}, we get $$u^{(n)}(r)=\frac{a^n+r-ar-1}{1-a}=\frac{a^n+b}{1-a},~\text{ for all }n\in \mathbb{N}.$$ From Remark 3.3 it follows that $r-ar-1=-a^t$, for some $t\in\mathbb{Z}$. It is clear that $t\neq 0$, as otherwise $r-ar=0,$ i.e., $r(1-a)=0,$ i.e., either $r=0$ or $a=1$, which is not true!
If $t=-n$ for some $n\in\mathbb{N}$, then $a^n(r-ar-1)=-1,$ again an impossibility as $|a|\ge 2$! Thus $t\in\mathbb{N}$ and $r(1-a)=1-a^t$, i.e., $r=1+\cdots+a^{t-1}$. So $t\ge 2$, $a|r-1$ and $u^{(t)}(r)=0$, i.e., $u$ is nilpotent at $r$, a contradiction! Thus $u(r)-r$ cannot be a unit and $u(r)=r\pm q_1^{s_1}\cdots q_k^{s_k}$, for a suitable collection of $s_i's$ in $\mathbb{N}\cup\{0\}$ with $\sum\limits_i s_i\ge 1$. So $b=r-ar\pm q_1^{s_1}\cdots q_k^{s_k}$. But then $b|r$ implies that $b|q_1^{s_1}\cdots q_k^{s_k}$, i.e., $\text{there exists}~t_i\in\mathbb{N}\cup \{0\}$, with $t_i\le s_i$ for every $i$, such that $b=\pm q_1^{t_1}\cdots q_k^{t_k}$. From $b=r-ar\pm q_1^{s_1}\cdots q_k^{s_k}$ we get $ra=r-b\pm q_1^{s_1}\cdots q_k^{s_k}=r-b(\pm 1\pm q_1^{s_1-t_1}\cdots q_k^{s_k-t_k})$, i.e., $r|b(\pm 1\pm q_1^{s_1-t_1}\cdots q_k^{s_k-t_k})$.\\
Suppose, if possible, all the $t_i's$ are zero. Then $b=\pm 1$ and so $r$ must divide $\pm 1\pm q_1^{s_1}\cdots q_k^{s_k}$, which is clearly absurd as $\gcd(r,\pm 1\pm q_1^{s_1}\cdots q_k^{s_k})=\gcd(q_1^{a_1}\cdots q_k^{a_k},\pm 1\pm q_1^{s_1}\cdots q_k^{s_k})=1$. Thus $\sum\limits_i t_i\ge 1$. All of these now boil down to the following two cases:

\subsubsection*{Case 1. $There~ exists~j\in\{1,\ldots,k\}~ s.t.~s_j>t_j$.}
Since $\gcd(r,\pm 1\pm q_1^{s_1-t_1}\cdots q_k^{s_k-t_k})=1$, $r$ must divide $b$ so that we can deduce $r=\pm b$. So $a_i=t_i\le s_i,~\text{ for all}~i\in\{1,\ldots,k\}$. We use the technique of reduction of polynomials here. Define $v=v(x):=\frac{1}{r}u(rx)=ax\pm 1.$ Then $$v(1)=\frac{1}{r}u(r)=\frac{1}{r}(r\pm q_1^{s_1}\cdots q_k^{s_k})=1\pm q_1^{s_1-a_1}\cdots q_k^{s_k-a_k}$$ and $v\in L_{1,A}^1\setminus N_1$. It follows now from the list in Corollary 5.2 that we have two possibilities for $v$ (since $|a|\ge 2$):
\begin{enumerate}
    
    \item[(i)] $v(x)=\pm q_1^{s_1-a_1}\cdots q_k^{s_k-a_k}x+1$, in which case $u(x)=\pm q_1^{s_1-a_1}\cdots q_k^{s_k-a_k}x+r$, or
    
    \item[(ii)] $v(x)=-2x-1$ (only when $2\in A$), in which case $u(x)=-2x-r$.
\end{enumerate}
One can check that both (i) and (ii) are indeed in $S_r$.

\subsubsection*{Case 2. $For~each~i\in\{1,\ldots, k\},~s_i=t_i$.}
Then $\pm q_1^{s_1}\cdots q_k^{s_k}=b|r$, i.e., $a_i\ge s_i$ for each $i$. 
 From $b=r-ar\pm q_1^{s_1}\cdots q_k^{s_k}$ we get $r(1-a)=\pm 2q_1^{s_1}\cdots q_k^{s_k}=\pm 2b$. Thus either $r=\pm b$ or $r=\pm 2b$. The first possibility has been taken care of in \textit{Case }1. So we can assume that $r=2q_1^{s_1}\cdots q_k^{s_k}=\pm 2b$. That would mean $a=2$ and that $2\in A$. Without loss of generality, let $q_1=2$ so that $r=2^{s_1+1}\cdots q_k^{s_k}$. Rewriting $b=r-ar\pm q_1^{s_1}\cdots q_k^{s_k}$ gives us $2r=r-2b$, i.e., $r=-2b$. This means that $u(x)=2x-\frac{r}{2}$. It follows from the \textit{linear iteration formula} that $$u^{(n)}(r)=\frac{r}{2}\cdot (2^n+1),~n\in\mathbb{N}.$$ Letting $\alpha=2,\beta =-1, \gamma=1$, we can see that neither $\frac{\beta}{\gamma}$ nor $\frac{\gamma}{\beta}$ is a power $\alpha$. Thus from the Lemma 3.2, $\mathcal{P}\setminus P(2^n+1)$ is an infinite set so that $\mathcal{P}\setminus P(\frac{r}{2}\cdot (2^n+1))$ is also an infinite set. So $2x-\frac{r}{2}\not\in S_r$. This completes the proof.\endproof{}
\end{proof}

 It follows directly from Fact 3.1 and Theorem 5.3 that:

\begin{corollary}
If $r=-q_1^{a_1}\cdots q_k^{a_k}$ is the prime decomposition of an integer $r\le -2$, then the following is the list of all polynomials in $S_r$\textup{:}
\begin{enumerate}
    \item[(1)] $x- q_1^{s_1}\cdots q_k^{s_k}$, where $s_i\in\mathbb{N}\cup\{0\}$.
    
    \item[(2)] $x+q_1^{s_1}\cdots q_k^{s_k}$, where $s_i\in\mathbb{N}\cup\{0\}$ with at least one $j\in\{1,\ldots, k\}$ s.t. $s_j>a_j$.
    
    \item[(3)] $\pm q_1^{s_1}\cdots q_k^{s_k}x-r$, where $s_i\in\mathbb{N}\cup\{0\}$ with $\sum\limits_{i} s_i\ge 1$.
    
    \item[(4)] $-2x+r$ (only when $r$ is even).
\end{enumerate}
\end{corollary}

\section{Some open problems}
For $u(x)$, a non-constant polynomial over $\mathbb{Z}$, let
$$N(u):=\{r\in\mathbb{Z}~|~u\in N_r \} ,~~LN(u):=\{r\in\mathbb{Z}~|~u\in L_{r,\emptyset} \}.$$ Then one can look at the following questions:

\begin{enumerate}
     \item[Q1.] Describe all $u's$ such that $N(u)$ is finite.
    
    \item[Q2.] Describe all $u's$ such that $LN(u)$ is finite.
    
    \item[Q3.] Given $r\in\mathbb{Z}$, describe all $u's$ such that $r\in LN(u)$.
\end{enumerate}

{\bf Few words about the open problems.} The methods used in the paper are actually useful for linear polynomials and due to Fact 1.1 we were able to achieve our goal by studying the linear polynomials only. It is clear that if $u$ is assumed to be linear, then $N(u)$ is finite for every $u$ except for the polynomials of the form $\pm x\pm c$, $c\in\mathbb{Z}\setminus\{0\}$. This answers Q1 for linear polynomials $u(x)$. It should be noted that $LN(u)=N(u)\cup\{r\in\mathbb{Z}~|~u\in S_r\}$. So for Q2, due to Fact 1.1 again, if the degree of $u$ is greater than or equal to 2 then $LN(u)$ is finite so that $N(u)$ is finite. For linear polynomials of the form $\pm x\pm c$, $c\in\mathbb{Z}\setminus\{0\}$, $N(u)$ is infinite and so $LN(u)$ is infinite. If $u=ax+b$ with $a\neq 1$, then we know that $u^{(n)}(r)=\frac{a^n(r-a-b)+b}{1-a}$. Note that $\frac{b}{a+b-r}=a^k$, for some $k\in\mathbb{Z}$ can only be true for finitely many values of $r$. So it follows from Lemma 3.2 that $\mathcal{P}\setminus\cup_{n\in\mathbb{N}} P((r-a-b)a^n+b)$ is finite for only finitely many values of $r$. Thus for these polynomials, $LN(u)$ is indeed finite. In this paper we have fully classified $S_r$, for every integer $r$. Also given an $r$, using the \textit{linear iteration formula} it should be easy enough to describe the linear $u's$ such that $r\in N(u)$. So we have a partial answer for Q3 here. Thus if one knows enough about the form of the nilpotent polynomials of degree higher than or equal to 2, one should potentially be able to answer all of the questions above. This is a work in progress and it is of course important enough to be considered as a separate work.

\noindent {\bf Acknowledgements.} The author gratefully acknowledges his advisors Prof. Alexander Borisov and Prof. Adrian Vasiu for their constant support, encouragement and very helpful suggestions. The author would also like to thank Prof. Jeremy Rouse for suggesting \cite{CS97} which was one of the main tools used to prove Theorems 5.1 and 5.3 and Prof. Kiran Kedlaya for maintaining a wonderful archive of William Lowell Putnam Mathematics Competition questions and answers here \url{https://kskedlaya.org/putnam-archive/}, which was very helpful in the proof of Theorem 4.4. Finally, the author would like to thank Prof. Thomas Tucker for his suggestion to use Theorem 5 of \cite{TS13}, which led to Fact 1.1 and that was very useful to boil down the possible candidates for locally nilpotent polynomials.

\end{document}